\documentclass[12pt,amssymb]{amsart}
\usepackage{amssymb}
\usepackage{amsmath}
\usepackage{enumerate}
\usepackage[mathscr]{eucal}

\newcommand{\im}{{\rm im}\:}

\newtheorem{thm}{Theorem}[section]
\newtheorem{lemma}[thm]{Lemma}
\newtheorem{prop}[thm]{Proposition}
\newtheorem{cor}[thm]{Corollary}
\theoremstyle{definition}
\newtheorem{remark}[thm]{Remark}

\begin{document}

\title[Congruence subgroup property]{The congruence subgroup property for $S$-arithmetic subgroups of simple algebraic groups when $S$ has positive Dirichlet density}

\author[A.S.~Rapinchuk]{Andrei S. Rapinchuk}

\begin{abstract}
Let $G$ be an absolutely almost simple simply connected algebraic group defined over a number field $K$, and let $M/K$ be the minimal Galois extension over which $G$ becomes 
an inner form of a split group. Assume that $G$ satisfies the Margulis-Platonov conjecture over $K$. We prove that if $S$ is a set of valuations of $K$ that contains all archimedean 
ones but does not contain any nonarchimedean valuations $v$ for which $G$ is anisotropic over the completion $K_v$ such that its intersection $S \cap \mathrm{Spl}(M/K)$ with the set $\mathrm{Spl}(M/K)$  of nonarchimedean valuations of $K$ that split completely in $M$ has positive Dirichlet density, then the congruence kernel $C^S(G)$ is trivial. This result provides additional evidence for 
Serre's Congruence Subgroup Conjecture \cite{Serre}. The proof does not involve any case-by-case considerations and relies on previous results concerning the congruence kernel \cite{PR-CSP} and recent results 
on almost strong approximation \cite{CW}, \cite{RT}. 
\end{abstract}

\address{Department of Mathematics, University of Virginia,
Charlottesville, VA 22904-4137, USA}

\email{asr3x@virginia.edu}

\maketitle

\section{Main Theorem}

Let $G$ be an absolutely almost simple simply connected algebraic group defined over a~number field $K$, let $S$ be a subset of the set $V^K$ of all (inequivalent) valuations of $K$ containing the set 
$V^K_{\infty}$ of archimedean valuations, and let $C^S(G)$ be the corresponding congruence kernel (cf. \cite{CSP-Sur1}, \cite{CSP-Sur2}).  Then the {\it Congruence Subgroup Problem} (CSP) amounts to computing $C^S(G)$ for given $G$ and $S$; recall that $C^S(G)$ is trivial if and only if 
the classical congruence subgroup property holds for the group $G(\mathcal{O}(S))$ of points of $G$ over the ring $\mathcal{O}(S)$ of $S$-integers in $K$ associated with any faithful $K$-defined representation $G \hookrightarrow \mathrm{GL}_n$ (see \S\ref{S:Centr}). J.-P.~Serre \cite{Serre} conjectured that for finite $S$ the congruence kernel  $C^S(G)$ should be finite is 
$$
\mathrm{rk}_S \: G \, := \, \sum_{v \in S} \mathrm{rk}_S \: G \, \geq \, 2 \ \ \text{(and} \ \ \mathrm{rk}_{K_v}\: G > 0 \ \ \text{for all} \ \ v \in S \setminus V^K_{\infty}),   
$$ 
where $\mathrm{rk}_P\: G$ denoted the rank of $G$ over a field extension $P$ of $K$. We note that if $C^S(G)$ is finite and the Margulis-Platonov conjecture (MP) concerning the normal subgroup structure of $G(K)$ (see below) holds,  the metaplectic kernel computations (see \cite{PR-Met}) 
provide a precise description of $C^S(G)$.

While Serre's conjecture has been confirmed in a number of cases (see \cite{CSP-Sur1}, \cite{CSP-Sur2} for a survey of available results), it remains completely open for some anisotropic groups that require noncommutative division algebras in their description, including even anisotropic groups of type $\textsf{A}_1$, i.e. norm 1 groups in quaternion division algebras. However, one can generate additional evidence for Serre's conjecture by investigating it for {\it infinite} $S$. More precisely, it follows from Serre's conjecture couple with our computations of the metaplectic kernel \cite{PR-Met} that $C^S(G)$ should be trivial for any $G$ as above and any infinite $S$ such that $\mathrm{rk}_{K_v}\: G > 0$ for $v \in S \setminus V^K_{\infty}$, and this is what one would like to prove. In \cite{PR-Met}, this was established for $S$ {\it cofinite} (i.e., when $V^K \setminus S$ is finite) assuming the truth of (MP). In \cite{PR-CSP}, the result was extended to sets $S$ of the form a generalized arithmetic progression minus a finite set. M.M.~Radhika and M.S.~Raghunathan \cite{RadRag} showed that for anisotropic inner forms of type $\textsf{A}_n$ (i.e., for groups of the form $\mathrm{SL}_{1 , D}$ where $D$ is a central division algebra over $K$) the results remains valid if $S$ is a generalized arithmetic progression minus any subset of Dirichlet density zero. Recently in \cite{RT}, using results involving a new property of {\it almost strong approximation}, the result of \cite{RadRag} was extended to all absolutely almost simple simply connected algebraic groups. Further progress on ASA by Y.~Cao and Y.~Wang \cite{CW} enables one to prove the following ultimate result in this direction. 

In order to give a precise statement, we first need to recall the statement of the Margulis-Platonov Conjecture: 

\vskip3mm 

\noindent (MP) \parbox[t]{15cm}{Let $\mathcal{A} = \{ v \in V^K_f := V^K \setminus V^K_{\infty} \ \vert \ \mathrm{rk}_{K_v}\: G = 0 \}$, set $\displaystyle 
G_{\mathcal{A}} := \prod_{v \in \mathcal{A}} G(K_v)$ and consider the diagonal map $\delta \colon G(K) \to G_{\mathcal{A}}$.  Then for every noncentral normal subgroup $N$ of $G(K)$ 
there should exist an open normal subgroup $W$ of $G_{\mathcal{A}}$ such that $N = \delta^{-1}(W)$. In particular, if $\mathcal{A} = \varnothing$ (which is automatically the case 
if $G$ is not of type $\textsf{A}_n$) then $G(K)$ does not have any proper noncentral normal subgroups.}

\vskip3mm 

\noindent For results on (MP) obtained prior to 1990 -- see \cite[Ch. IX]{PlRa}. Subsequently, (MP) was proved also for all anisotropic inner forms of type $\textsf{A}_n$ -- see \cite{R-MP}, \cite{Segev}.

Next, for a finite extension $L/K$ we let $\mathrm{Spl}(L/K)$ denote the set of $v \in V^K_f$ that split completely in $L$ (cf. \cite[p. 49]{Neu}). Also, for a set of places $T \subset V^K_f$, we let $\mathfrak{d}_K(T)$ denote the Dirichlet density of $T$ (i.e., the Dirichlet density of the corresponding set of primes), cf. \cite[Ch. VII, \S13]{Neu} and \S\ref{S:Dens} below. With these notations, we have the following theorem which is the main result of this paper.  
\begin{thm}\label{T:Main} 
Let $G$ be an absolutely almost simple simply connected algebraic group over a number field $K$, and let $M$ be the minimal Galois extension of $K$ over which $G$ becomes an inner form of the split group. 
Assume that (MP) holds for $G$ over $K$. If $S \subset V^K$  contains $V^K_{\infty}$, does not contain any $v \in V^K_f$ with $\mathrm{rk}_{K_v}\: G = 0$ and satisfies  $\mathfrak{d}_K(S \cap \mathrm{Spl}(M/K)) > 0$ then $C^S(G) = \{ 1 \}$. 
\end{thm}

(It should be noted  that if $G$ is an {\it inner} form of the split group then $M=k$, hence $\mathrm{Spl}(M/K) = V^K_f$, so our assumption becomes $\mathfrak{d}_K(S \cap V^K_f) > 0$.)  

\vskip1.5mm 

The proof of the theorem, which we give in \ref{S:Proof}, has two major ingredients. First, we use the conditions for the centrality of the congruence kernel developed in \cite{PR-CSP}. We review this material in \ref{S:Centr} and derive a condition for centrality in the form needed for the proof of Theorem~\ref{T:Main}. Another ingredient is  almost strong approximation (ASA) introduced in \cite{RT}. The result on ASA that we need (Theorem \ref{T:ASA-tori}) is a reformulation of the results from \cite{CW}, \cite{RT} in terms of {\it upper} Dirichlet density, which is defined and discussed in \S\ref{S:Dens}. We will also use our previous results on the existence of generic tori -- see \cite{PR-Gen1}, \cite{PR-WC}, \cite{PR-Gen2}. Finally, it should be pointed out that our proof of Theorem \ref{T:Main} in the situation where $M \neq K$ (i.e., when $G$ is an outer form) relies on the previous results from \cite{PR-CSP}, \cite{RT} covering the case where $S$ almost contains a generalized arithmetic progression (cf. Case 1 in \S\ref{S:Proof}).

\section{A condition for centrality of the congruence kernel}\label{S:Centr} 

We begin with a brief review of the fundamental notions associated with the Congruence Subgroup Problem (see \cite{CSP-Sur1}, \cite{CSP-Sur2} for more details). While these can be defined for arbitrary linear algebraic groups, we will focus here on the case of principal interest to us where $G$ is an absolutely almost simple simply connected algebraic group defined over a number field $K$. Let us fix for now a faithful $K$-defined representation $\rho \colon G \hookrightarrow \mathrm{GL}_n$. 

Next, given a subset $S \subset V^K$ containing $V^K_{\infty}$, we let $\mathcal{O}(S)$ denote the ring of $S$-integers in $K$, and consider the group of $\mathcal{O}(S)$-points 
$$
G(\mathcal{O}(S)) := G(K) \cap \rho^{-1}\left( \mathrm{GL}_n(\mathcal{O}(S)) \right). 
$$
Furthermore, for a nonzero ideal $\mathfrak{a} \subset \mathcal{O}(S)$, we define the corresponding congruence subgroup 
$$
G(\mathcal{O}(S) , \mathfrak{a}) := G(\mathcal{O}(S)) \cap \rho^{-1}\left( \mathrm{GL}_n(\mathcal{O}(S) , \mathfrak{a}) \right), 
$$
where $\mathrm{GL}_n(\mathcal{O}(S) , \mathfrak{a})$ is the usual congruence subgroup of $\mathrm{GL}_n(\mathcal{O}(S))$ of level $\mathfrak{a}$. Then for any ideal $\mathfrak{a}$, the congruence subgroup $G(\mathcal{O}(S) , \mathfrak{a})$ is a normal subgroup of $G(\mathcal{O}(S))$ of finite index, and the classical Congruence Subgroup Problem asks if the converse is true, viz. if every finite index normal subgroup of $G(\mathcal{O}(S))$ contains a suitable congruence subgroup. The current analysis of this problem relies on the notion of the ($S$-)congruence kernel, introduced by J.-P.~Serre \cite{Serre}, which we now recall. 
One defines two topologies on the group of $K$-rational points $G(K)$ compatible with the group structure and independent of the choice of $\rho$: 

\vskip1.5mm 

\noindent $\bullet$ \parbox[t]{15cm}{$\tau_a^S$, the $S$-arithmetic topology, having the collection of all finite index (normal) subgroups $N \subset G(\mathcal{O}(S))$ as a fundamental system of neighborhoods of the identity, and}

\vskip2mm 

\noindent $\bullet$ \parbox[t]{15cm}{$\tau_c^S$, the $S$-congruence topology, having the collection of all congruence subgroups $G(\mathcal{O}(S) , \mathfrak{a})$ as a fundamental system of neighborhoods 
of the identity.}

\vskip2mm 

\noindent Then the affirmative answer to the classical CSP for $G(\mathcal{O}(S))$ is equivalent to the fact that $\tau_a^S = \tau_c^S$. To gauge the difference between these two topologies in the general case, one observes that $G(K)$ admits completions in the category of locally compact topological groups with respect to both $\tau_a^S$ and $\tau_c^S$ which are called the $S$-arithmetic completion and the $S$-congruence completion and denoted $\widehat{G}^S$ and $\overline{G}^S$, respectively. Since $\tau_a^S$ is stronger than $\tau_c^S$, there exists a continuous group homomorphism $\pi \colon \widehat{G}^S \to \overline{G}^S$ which turns out to be surjective. Thus, we have the following exact sequence
\begin{equation}\tag{C}\label{E:F1} 
1 \to C^S(G) \longrightarrow \widehat{G}^S \stackrel{\pi}{\longrightarrow} \overline{G}^S \to 1, 
\end{equation}
where $C^S(G) := \ker \pi$ is the ($S$-){\it congruence kernel} of $G$ (note that it is independent of $\rho$). It is easy to see that $\tau_a^S = \tau_c^S$ if and only if $C^S(G) = \{ 1 \}$. Thus, Theorem \ref{T:Main} asserts that in the situation described in its statement we have the affirmative answer to the classical CSP for $G(\mathcal{O}(S))$. If the group $G(\mathcal{O}(S))$ is finite then trivially $C^S(G) = \{ 1 \}$ (in fact, this situation never occurs in the case of infinite $S$ which is our main focus in this note). Otherwise it follows from the Strong Approximation Theorem (cf. \cite[Ch. VII]{PlRa}) that  the completion $\overline{G}^S$ can be naturally identified with the group of $S$-adeles $G(\mathbb{A}_S)$ where $\mathbb{A}_S = \mathbb{A}_{K , S}$ is the ring of $S$-adeles of $K$, and we will routinely use this identification.  

A standard strategy for proving that $C^S(G) = \{ 1 \}$ is based on proving first that $C^S(G)$ is {\it central} (i.e. (\ref{E:F1}) is a central extension). Once this is accomplished, one shows, assuming that (MP) holds for $G(K)$,  that $C^S(G)$ is isomorphic to the (dual of the) {\it metaplectic kernel} $M(S , G)$. It follows from the computations of $M(S , G)$ in \cite{PR-Met} that for any infinite $S$, the metaplectic kernel $M(S , G)$ is trivial. Thus, in order to prove Theorem \ref{T:Main} it is enough to show that in the situation described in it the congruence kernel $C^S(G)$ is central. For this we will use the following statement (which is a variant of Theorem 3.1(ii) in \cite{PR-CSP}) where for the ease of notations we let $\Gamma$ denote $G(\mathcal{O}(S))$.  
\begin{thm}\label{T:0001}
Assume that $G$ satisfies (MP) over $K$ and that a subset $S \subset V^K$ contains $V^K_{\infty}$ and does not contain any $v \in V^K_f$ with $\mathrm{rk}_{K_v}\: G = 0$. If there exists $n \geq 1$ such that for every $x \in C^S(G)$ we have 
\begin{equation}\label{E:Cond1}
\pi(Z_{\widehat{G}^S}(x)) \supset \Gamma^n, 
\end{equation}
where $Z_{\widehat{G}^S}(x)$ denotes the centralizer of $x$ in $\widehat{G}^S$. Then (\ref{E:F1}) is a central extension. 
\end{thm}

It should be pointed that at least in spirit our condition (\ref{E:Cond1}) is akin to the considerations used by Serre \cite{Serre} in the context of the CSP for $G = \mathrm{SL}_2$ over the ring $\mathcal{O}(S)$ with infinite group of units $\mathcal{O}(S)^{\times}$, where he first proves that for a unit $u \in \mathcal{O}(S)$ of infinite order and $h(u) = \mathrm{diag}(u , u^{-1})$ there exists $n \geq 1$ such that $h(u)^n$ centralizes the congruence kernel $C^S(G)$, and then derives the centrality of the latter. (In fact, Serre showed that in this case one can take for such an $n$ the number of roots of unity in $K$, however 
in our application of Theorem \ref{T:0001} for proving Theorem \ref{T:Main} the value of $n$ is less explicit -- cf. the beginning of the proof of Proposition \ref{P:Main}).   

Theorem \ref{T:0001} follows from the following two results established in \cite{PR-CSP}. Let $\mathcal{A}$ be as in the statement of (MP) noting that  by our assumption $\mathcal{A} \cap S = \varnothing$. Set 
$$
\mathfrak{Z} = \{ v \in V^K \setminus (\mathcal{A} \cup S) \, \vert \, \pi(Z_{\widehat{G}^S}) \supset G(K_v) \} 
$$ 
(hereinafter, we routinely identify $\overline{G}^S$ with $G(\mathbb{A}_S)$, which enables us to naturally consider $G(K_v)$ for $v \notin S$ as a subgroup of $\overline{G}^S$). 

\begin{prop}\label{P:0001} {\rm (Proposition 2.2 in \cite{PR-CSP})}  
Let $v \in V^K \setminus (\mathcal{A} \cup S)$. Assume that there exists $a \in G(K_v)$ such that $a \in \pi(Z_{\widehat{G}^S}(x))$ for every $x \in C^S(G)$. Then $v \in \mathfrak{Z}$. 
\end{prop}

\begin{prop}\label{P:0002} {\rm (cf. Proposition 2.5 in \cite{PR-CSP})} 
Assume that $G$ satisfies (MP) over $K$ and that $\mathcal{A} \cap S = \varnothing$. If $\mathfrak{Z} = V^K \setminus (\mathcal{A} \cup S)$ then (\ref{E:F1}) is a central extension. 
\end{prop}

\vskip2mm 

{\it Proof of Theorem \ref{T:0001}.} Since $\ker \pi = C^S(G)$ is a profinite, hence compact, group, the map $\pi$ is closed, and therefore (\ref{E:Cond1}) implies that for every $x \in C^S(G)$ we have the inclusion 
$$
\pi\left( Z_{\widehat{G}^S}  \right) \supset \overline{\Gamma}^n, 
$$ 
where $\overline{\Gamma}$ is the closure of $\Gamma$ in $\overline{G}^S = G(\mathbb{A}_S)$. Clearly, $W := \prod_{v \notin S} G(\mathcal{O}_v)$ is an open compact subgroup of $G(\mathbb{A}_S)$, and since $G(K)$ is dense in $G(\mathbb{A}_S)$ by the Strong Approximation Theorem, we obtain that $\Gamma = G(K) \cap W$ is dense in $W$, and therefore $\overline{\Gamma} = W$. For each $v \notin S$, the group $G(\mathcal{O}_v)$ is Zariski-dense in $G$ (cf. \cite[Lemma 3.2]{PlRa}), and therefore $G(\mathcal{O}_v)^n \not\subset Z(G(K_v)$. Applying Proposition \ref{P:0001}, we obtain that $v \in \mathfrak{Z}$, and thus $\mathfrak{Z} = V^K \setminus (\mathcal{A} \cup S)$. Then the centrality of (\ref{E:F1}) follows from Proposition \ref{P:0002}. \hfill $\Box$ 

\vskip1mm 

The following statement is used to verify condition (\ref{E:Cond1}). 
\begin{prop}\label{P:0003}
Let $n \geq 1$. Assume that for every finite index normal subgroup $N$ of $\Gamma$ and any $y \in \overline{N} \cap \Gamma$ we have the inclusion 
\begin{equation}\label{E:Cond2}
Z(N, y) (\overline{N} \cap \Gamma) \supset \Gamma^n, 
\end{equation}
where $Z(N, y) := \{ t \in \Gamma \ \vert \ [y , t] \in N \}$. Then (\ref{E:Cond1}) holds for all $x \in C^S(G)$.   
\end{prop}

(We note that $Z(N, y)$ is simply the pullback of the centralizer of $yN$ in $\Gamma/N$ under the canonical homomorphism $\Gamma \to \Gamma/N$.)

\begin{proof}
(cf. \cite[proof of Proposition 3.2]{PR-CSP}) Fix $x \in C^S(G)$ and $\gamma \in \Gamma$. Let $\mathcal{R}$ be the family of all open normal subgroups of the profinite completion $\widehat{\Gamma}$, which is naturally identified with the closure of $\Gamma$ in $\widehat{G}^S$. For $R \in \mathcal{R}$, we set 
$$
N_R = \Gamma \cap R \ \ \text{and} \ \ \tilde{R} = \pi^{-1}(\overline{N_R}) = R \cdot C^S(G), 
$$
and pick an arbitrary $x_R \in \Gamma \cap (xR)$. Applying (\ref{E:Cond2}) to $N = N_R$ and $y = x_R$, we obtain that 
$$
\gamma^n \in Z(N_R , x_R) (\overline{N_R} \cap \Gamma) \subset \tilde{Z}(R, x) \tilde{R}, 
$$
where 
$$
\tilde{Z}(R , x) := \{ t \in \widehat{\Gamma} \ \vert \ [x , t] \in R \} = \{ t \in \widehat{\Gamma} \ \vert \ [x_R , t] \in R \}, 
$$
hence
\begin{equation}\label{E:Cond3} 
\gamma^n \tilde{R} \cap \tilde{Z}(R, x) \neq \varnothing \ \ \text{for any} \ \ R \in \mathcal{R}. 
\end{equation}
Using the compactness argument, one easily derives from (\ref{E:Cond3}) that 
\begin{equation}\label{E:Cond4}
\gamma^n C^S(G) \cap Z_{\widehat{\Gamma}}(x) \neq \varnothing, 
\end{equation}
which implies that $\gamma^n \in \pi(Z_{\widehat{\Gamma}}(x))$, yielding (\ref{E:Cond1}). 

To prove (\ref{E:Cond4}), one observes that 
$$
\bigcap_{R \in \mathcal{R}} \tilde{R} = C^S(G) \ \ \text{and} \ \ \bigcap \tilde{Z}(R , x) = Z_{\widehat{\Gamma}}(x),  
$$
and that for any $R_1, \ldots , R_d \in \mathcal{R}$, we have 
$$
\tilde{Z}(R_1, x) \cap \cdots \cap \tilde{Z}(R_d, x) = \tilde{Z}(R_1 \cap \cdots \cap R_d, x). 
$$
So, if (\ref{E:Cond4}) does not hold, then due to the compactness of $C^S(G)$, one can find $R' \in mathcal{R}$ such that $\gamma^nC^S(G) \cap \tilde{Z}(R', x) = \varnothing$. Next, being a closed 
subgroup of $\widehat{\Gamma}$, the group $\tilde{Z}(R', x)$ is also compact, and hence there exists $R'' \in \mathcal{R}$ such that $\gamma^nR'' \cap \tilde{Z}(R', x) = \varnothing$. Then for $R = R' \cap R''$, the condition (\ref{E:Cond3}) fails to hold, a contradiction. 
\end{proof}

Combining Theorem \ref{T:0001} and Proposition \ref{P:0003}, we obtain the following. 
\begin{cor}\label{C:0001} 
Assume that (MP) holds for $G$ over $K$ and that a subset $S \subset V^K$ contains $V^K_{\infty}$ and satisfies $\mathcal{A} \cap S = \varnothing$. If there 
exists $n \geq 1$ such that for every finite index normal subgroup $N$ of $\Gamma$ and any $y \in \overline{N} \cap \Gamma$ we have the inclusion (\ref{E:Cond2}),  then $C^S(G)$ is central. 
\end{cor}

\section{On Dirichlet density}\label{S:Dens} 

\subsection{Definitions and elementary properties} Let $K$ be a number field, $\mathcal{O}$ be its ring of algebraic integers, and $\mathcal{P}$ be the set of primes, i.e. the set of nonzero prime ideals of $\mathcal{O}$. (In the sequel, we will freely identify $\mathcal{P}$ with $V^K_f$.) As usual, for $\mathfrak{p} \in \mathcal{P}$, we let $\mathrm{N}\mathfrak{p}$ denote the number of elements in the quotient $\mathcal{O}/\mathfrak{p}$. Gven a subset $A \subset \mathcal{P}$, we define 
$$
\xi_{K,A}(s) = \sum_{\mathfrak{p} \in A} \frac{1}{(\mathrm{N}\mathfrak{p})^s} \ \ \text{for real} \ \ s > 1 
$$
(it is well-known that the series converges). We recall (cf. \cite[Ch. VII, \S13]{Neu}) that the {\it Dirichlet density} of $A$ is defined to be 
$$
\mathfrak{d}_K(A) = \lim_{s \to 1+0} \frac{\xi_{K,A}(s)}{\log((s - 1)^{-1})}
$$
if the limit in the right-hand side exists. We also define the {\it upper Dirichlet density} of $A$ by 
$$
\bar{\mathfrak{d}}_K(A) = \varlimsup_{s \to 1+0} \frac{\xi_{K,A}(s)}{\log((s - 1)^{-1})}.
$$
We note that $\bar{\mathfrak{d}}_K(A)$ is defined for all subsets $A \subset \mathcal{P}$, and is equal to the usual Dirichlet density 
$\mathfrak{d}_K(A)$ if the latter exists. We also recall that given a finite Galois extension $L/K$, by Chebotarev's Density Theorem, 
the set $\mathrm{Spl}(L/K)$ of primes of $K$ that set completely in $L$ has Dirichlet density 
$$
\mathfrak{d}_K(\mathrm{Spl}(L/K)) = \frac{1}{[L : K]}. 
$$  

\begin{lemma}\label{L:Dens1}

\begin{enumerate}[ \rm (1)]

\item For any two subsets $A , B \subset \mathcal{P}$, we have $\bar{\mathfrak{d}}_K(A \cup B) \leq \bar{\mathfrak{d}}_K(A) + \bar{\mathfrak{d}}_K(B)$.

\vskip1mm 

\item If $A_1, \ldots , A_r \subset \mathcal{P}$ are subsets such that $\bar{\mathfrak{d}}_K(A_1 \cup \cdots \cup A_r) \geq \varepsilon > 0$ then there exists 
$i \in \{1, \ldots , r\}$ such that $\bar{\mathfrak{d}}_K(A_i) \geq \varepsilon/r$. 

\end{enumerate}
\end{lemma}
\begin{proof}
(1): We obviously have 
$$
\xi_{A \cup B}(s) \leq \xi_A(s) + \xi_B(s) \ \ \text{for} \ \ s > 1.  
$$
Dividing by $\log((s - 1)^{-1})$ and taking upper limits (using the sum rule for these), we obtain our claim. 

\vskip1mm 

(2): Indeed, if we assume that $\bar{\mathfrak{d}}_K(A_i) < \varepsilon/r$ for all $i$, then applying part (1) repeatedly, we would obtain that $\bar{\mathfrak{d}}_K(A_1 \cup \cdots \cup A_r) < \varepsilon$, 
a contradiction. 
\end{proof}

\begin{lemma}\label{L:Dens2}
Suppose subsets  $A , B \subset \mathcal{P}$ have Dirichlet density and are contained in a subset $C \subset \mathcal{P}$ that also has Dirichlet 
density. Then 
$$
\bar{\mathfrak{d}}_K(A \cap B) \geq \mathfrak{d}_K(A) + \mathfrak{d}_K(B) - \mathfrak{d}_K(C).
$$
\end{lemma}
\begin{proof}
We have 
$$
\xi_{A \cap B}(s) = \xi_A(s) + \xi_B(s) - \xi_{A \cup B}(s) \geq \xi_A(s) + \xi_B(s) - \xi_{C}(s).
$$
Dividing by $\log((s - 1)^{-1})$ and taking upper limits as $s \to 1+0$ we obtain our claim since the usual limit of the ratio 
in the right-hand side exists and equals $\mathfrak{d}_K(A) + \mathfrak{d}_K(B) - \mathfrak{d}_K(C)$. 
\end{proof}

\begin{lemma}\label{L:Dens3} 
Let $A_0, A_1, \ldots , A_r$ be subsets of $\mathcal{P}$ that have Dirichlet density, and assume that all sets are contained in a subset $C \subset \mathcal{P}$ that also has 
Dirichlet density and that the union $A_1 \cup \cdots \cup A_r$ 
has Dirichlet density as well. If 
$$
\theta := \mathfrak{d}_K(A_0) + \mathfrak{d}_K(A_1 \cup \cdots \cup A_r) - \mathfrak{d}_K(C),  
$$
then there exists $i \in \{1, \ldots , r \}$ such that $\bar{\mathfrak{d}}_K(A_0 \cap A_i) \geq \theta/r$.   
\end{lemma}
\begin{proof}
By Lemma \ref{L:Dens2}, we have 
$$
\bar{\mathfrak{d}}_K(A_0 \cap (A_1 \cup \cdots \cup A_r)) \geq \mathfrak{d}_K(A_0) + \mathfrak{d}_K(A_1 \cup \cdots \cup A_r) - \mathfrak{d}_K(C) = \theta. 
$$
Thus, $\bar{\mathfrak{d}}_K((A_0 \cap A_1) \cup \cdots \cup (A_0 \cap A_r)) \geq \theta$, so our claim follows from Lemma \ref{L:Dens1}(2). 
\end{proof}

\begin{lemma}\label{L:Dens4}
Let $A_1, \ldots A_r$ be subsets of $\mathcal{P}$, and set $A = A_1 \cup \cdots \cup A_r$.  \vskip1mm 
\begin{enumerate}[\rm (1)]
  
\item $\displaystyle \xi_A(s) = \sum_{1 \leq i \leq r} \xi_{A_i}(s)- \sum_{1 \leq i_1 < i_2 \leq r} \xi_{A_{i_1} \cap A_{i_2}}(s) + \cdots + (-1)^{r-1} \xi_{A_1 \cap \cdots \cap A_r}(s)$;   

\vskip1mm 

\item if all intersections $A_{i_1} \cap \cdots \cap A_{i_{\ell}}$ with $i_1, \ldots , i_{\ell} \in \{1, \ldots , r\}$ and $1 \leq \ell \leq r$ have Dirichlet density then so does $A$ and 
$$
\mathfrak{d}_K(A) = \sum_{1 \leq i \leq r} \mathfrak{d}_K(A_i) - \sum_{1 \leq i_1 < i_2 \leq r} \mathfrak{d}_K(A_{i_1} \cap A_{i_2}) + \cdots + (-1)^{r-1}\mathfrak{d}_K(A_1 \cap \cdots \cap A_r). 
$$
\end{enumerate}
\end{lemma}
\begin{proof}
(1): Induction on $r$. The claim is obvious for $r = 1$ and is easy to prove for $r = 2$. Let now $r > 2$. Using the case $r = 2$ and the induction hypothesis, we obtain  
$$
\xi_A(s) = \xi_{A_1 \cup \cdots \cup A_{r-1}}(s) + \xi_{A_r}(s) - \xi_{(A_1 \cap A_r) \cup \cdots \cup (A_{r-1} \cap A_r)}(s) =  
$$
$$
\sum_{1 \leq i \leq r-1} \xi_{A_i}(s) - \sum_{1 \leq i_1 < i_2 \leq r-1} \xi_{A_{i_1} \cap A_{i_2}}(s) + \cdots + (-1)^{r-2} \xi_{A_1 \cap \cdots \cap A_{r-1}}(s) + \xi_{A_r}(s) - 
$$
$$
\left( \sum_{1 \leq i \leq r-1} \xi_{A_i \cap A_r}(s) - \sum_{1 \leq i_1 < i_2 \leq r-1} \xi_{A_{i_1} \cap A_{i_2} \cap A_{r}}(s) + \cdots + (-1)^{r-2} \xi_{A_1 \cap \cdots \cap A_r}(s)   \right) = 
$$
$$
\sum_{1 \leq i \leq r} \xi_{A_i}(s) - \sum_{1 \leq i_1 < i_2 \leq r} \xi_{A_{i_1} \cap A_{i_2}}(s) + \cdots + (-1)^{r-1} \xi_{A_1 \cap \cdots \cap A_r}(s), 
$$
as required. 

\vskip1mm 

(2): Dividing the expression for $\xi_A(s)$ from part (1) by $\log((s - 1)^{-1})$ and taking the limit as $s \to 1+0$, we obtain our claim. 
\end{proof}

\vspace{-5mm} 

\subsection{An application} We will now apply these results to a situation that arises in the proof of Theorem \ref{T:Main}. Let $M$ be a finite Galois extension of $K$ of degree $m$, and let $P_1, \ldots , P_r$ be finite Galois extensions of $K$ such that each $P_i$ contain $M$ and has  degree $t$ over $M$. We assume in addition that $P_1, \ldots , P_r$ are linearly disjoint over $M$, i.e. the tensor product $P_1 \otimes_M \cdots \otimes_M P_r$ is a field. Then for any $i_1, \ldots , i_{\ell} \in \{1, \ldots , r \}$, the degree of the compositum 
$$
[P_{i_1} \cdots P_{i_{\ell}} : K] = m \cdot t^{\ell}. 
$$ 
Set $A_i = \mathrm{Spl}(P_i/K)$, observing that $A_{i_1} \cap \cdots \cap A_{i_{\ell}} = \mathrm{Spl}(P_{i_1} \cdots P_{i_{\ell}}/K)$. It follows from Chebotarev's Theorem that the relevant Dirichlet densities have the following values: 
$$
\mathfrak{d}_K(A_i) = m^{-1} t^{-1} \ \ \text{and} \ \ \mathfrak{d}_K(A_{i_1} \cap \cdots \cap A_{i_{\ell}}) = m^{-1} t^{-\ell}. 
$$
It follows from Lemma \ref{L:Dens4}(2) that the union $A = A_1 \cup \cdots \cup A_r$ has Dirichlet density which equals 
\begin{equation}\label{E:Dens1}
\mathfrak{d}_K(A) = r \cdot m^{-1} t^{-1} - \binom{r}{2} \cdot m^{-1} t^{-2} + \cdots + (-1)^{r-1} m^{-1} t^{-r} = m^{-1} \cdot \left( 1 - \left( 1 - t^{-1}  \right)^r \right). 
\end{equation}
\begin{prop}\label{P:Dens1}
Keep the notations and conventions of the current subsection. Let $S \subset V^K$ be a subset such that $S \cap \mathrm{Spl}(M/K)$ has Dirichlet density, and set 
$$
\theta := \mathfrak{d}_K(S \cap \mathrm{Spl}(M/K)) - m^{-1} \cdot (1 - t^{-1})^r. 
$$
Then there exists $i \in \{ 1, \ldots , r \}$ such that $\bar{\mathfrak{d}}_K(S \cap \mathrm{Spl}(P_i/K)) \geq \theta/r$.   
\end{prop}
\begin{proof}
Set 
$$
A_0 = S \cap \mathrm{Spl}(M/K), \ \ A_i = \mathrm{Spl}(P_i/K) \ \ \text{for} \ \ i = 1, \ldots , r, 
$$
and $A = A_1 \cup \cdots A_r$. Clearly, $S \cap A_i = A_0 \cap A_i$ for all $i = 1, \ldots , r$. Furthermore, the subsets $A_0, A_1, \ldots , A_r$ are all contained in $C := \mathrm{Spl}(M/K)$, 
which has Dirichlet density $m^{-1}$. In view of (\ref{E:Dens1}), we have 
$$
\mathfrak{d}_K(A_0) + \mathfrak{d}_K(A) - \mathfrak{d}_K(C) = \mathfrak{d}_K(S \cap \mathrm{Spl}(M/K)) - m^{-1} \cdot (1 - t^{-1})^r = \theta. 
$$
So, our claim follows from Lemma \ref{L:Dens3}. 
\end{proof}

\section{On almost strong approximation in tori}\label{S:ASA}

One of the key ingredients of the proof of Theorem \ref{T:Main} is a result on {\it almost strong approximation}. This concept was introduced in \cite{RT}: Given a linear algebraic group $G$ defined over a number field $K$, we say that $G$ has almost strong approximation (ASA) with respect to a subset $S \subset V^K$ if the index $[G(\mathbb{A}_S) : \overline{G(K)}]$ is finite (where $\overline{\begin{array}{c} \  \end{array}}$  denotes the closure in the $S$-adelic topology). It is well-known (cf. \cite[Ch. VII]{PlRa}) that non-simply connected groups, in particular, tori, never have ASA for any finite $S$. At the same time, it was shown already in the initial version of \cite{RT} that tori and also general reductive groups do have ASA (under appropriate assumptions) with respect to certain infinite sets of valuations called in loc. cit. {\it tractable}\footnote{A tractable subset is a set of the form $V^K_{\infty} \cup (\mathcal{P}(L/K , \mathcal{C}) \setminus \mathcal{P}_0$ where $\mathcal{P}(L/K , \mathcal{C})$ is a generalized arithmetic progression associated with a finite Galois extension $L/K$ and a conjugacy class $\mathcal{C}$ in its Galois group $G(L/K)$ and $\mathcal{P}_0$ is a set of primes having Dirichlet density zero. We also recall that $\mathcal{P}(L/K , \mathcal{C})$ consists of those primes for which the corresponding valuation $v$ is unramified in $L$ and the Frobenius automorphism $\mathrm{Fr}(w\vert v)$ belongs to $\mathcal{C}$ for some (equivalently, any) extension $ w \vert v$.}. As an application we proved Theorem \ref{T:Main} in \cite{RT} for tractable $S$ (which always have positive Dirichlet density) -- we note that we in fact use this particular case in our proof of this theorem in the general situation, see \S\ref{S:Proof}.  Subsequently, Y.~Cao and Y.~Wang \cite{CW} used the work of C.~Demarche  \cite{Dem1}, \cite{Dem2} to extend the results from \cite{RT} on ASA (but not on the  CSP!) to general sets of positive Dirichlet density. We then pointed out in the later version of \cite{RT} that this generalization can also be achieved by our techniques with minimal adjustments required only in the treatment of quasi-split tori as the ``d\'evissage'' part of our argument goes through without any changes. We will use this approach to prove the following version of the result on ASA for tori needed in our proof of Theorem \ref{T:Main}, which is stated in terms of upper Dirichlet density.   

Let $\phi(\delta , d)$ be a function defined for $0 < \delta \leq 1$ and $d \in \mathbb{N}$ with values in $\mathbb{N}$. We say that $\phi$ is {\it super-decreasing} if $\delta_1 \leq \delta_2$ implies that $\phi(\delta_2 , d)$ divides $\phi(\delta_1 , d)$ for all $d \in \mathbb{N}$. 
\begin{thm}\label{T:ASA-tori}
There exists a super-decreasing function $\phi(\delta , d)$ such that given a $d$-dimensional torus $T$ defined over a number field $K$ with the minimal splitting field $P$ and a subset $S \subset V^K$ containing $V_{\infty}^K$ for which $\bar{\mathfrak{d}}_K(S \cap \mathrm{Spl}(P/K)) =: \delta > 0$, the index 
$$
[T(\mathbb{A}_S) : \overline{T(K)}] \ \ \text{divides} \ \ \phi(\delta , d), 
$$ 
and in particular $T$ has ASA with respect to $S$.
\end{thm}
\begin{proof}
First, let $T = \mathrm{R}_{P/K}(\mathbb{G}_m)$ where $P$ is a finite Galois extension of $K$. We then have compatible isomorphisms 
$$
T(\mathbb{A}_S) \simeq \mathbb{I}_{P, \overline{S}} \ \ \text{and} \ \ T(K) \simeq P^{\times} 
$$
where $\overline{S}$ consists of all extensions of valuations from $S$ to $P$, and $\mathbb{I}_{P , \overline{S}}$ is the group of $\overline{S}$-ideles of $P$. It is easy to see 
(cf. \cite[\S 3]{RT}) that is $S_0 = S \cap \mathrm{Spl}(P/K)$ and $\overline{S}_0$ consists of all extensions of valuations from $S_0$ to $P$ then 
$$
\bar{\mathfrak{d}}_P(\overline{S}_0) = [P : K] \cdot \bar{\mathfrak{d}}_K(S_0). 
$$
Thus, 
\begin{equation}\label{E:0401}
\bar{\mathfrak{d}}_P(\overline{S} \cap V^P_f) \geq [P : K] \cdot \bar{\mathfrak{d}}_K(S_0) \geq \bar{\mathfrak{d}}_K(S_0) = \delta.  
\end{equation}
Let us now show that 
\begin{equation}\label{E:Index01} 
[T(\mathbb{A}_S) : \overline{T(K)}] = [\mathbb{I}_{P , \overline{S}} : \overline{P^{\times}}] \leq \bar{\mathfrak{d}}_P(\overline{S} \cap V^P_f)^{-1} \leq \delta^{-1}. 
\end{equation}
The argument almost verbatim repeats the proof of Proposition 3.2 in \cite{RT}. Indeed, it is shown in \cite[proof of Proposition 3.1]{RT} that it is enough to demonstrate that for any open subgroup $B$ of 
$\mathbb{I}_{P , \overline{S}}$ containing $P^{\times}$, the (automatically finite) index $[\mathbb{I}_{P , \overline{S}} : B]$ is $\leq \bar{\mathfrak{d}}_P(\overline{S} \cap V^P_f)^{-1}$. 

Let $\pi_{\overline{S}} \colon \mathbb{I}_P \to \mathbb{I}_{P , \overline{S}}$ be the natural projection, and let $D = \pi_{\overline{S}}^{-1}(B)$. By the Existence Theorem of Global Class Field Theory (cf. \cite[Ch. VII]{ANT}), there exists an abelian extension $R/P$ of degree 
$$
m = [\mathbb{I}_P : D] = [\mathbb{I}_{P , \overline{S}} : B] 
$$
such that $N_{R/P}(\mathbb{I}_R) P^{\times} = D$. Since by construction $P_w^{\times} \subset D$ for all $w \in \overline{S}$ we conclude that 
$$
\overline{S} \cap V^P_f \subset \mathrm{Spl}(R/P). 
$$
According to Chebotarev's Density Theorem, $\mathfrak{d}_P(\mathrm{Spl}(R/P)) = m^{-1}$, so the above inclusion yields 
$$
\bar{\mathfrak{d}}_P(\overline{S} \cap V^P_f) \leq \bar{\mathfrak{d}}_P(\mathrm{Spl}(R/P)) = \mathfrak{d}_P(\mathrm{Spl}(R/P)) = m^{-1}. 
$$
Comparing with (\ref{E:0401}), we obtain (\ref{E:Index01}). Then the index $[T(\mathbb{A}_S) : \overline{T(K)}]$ divides 
$$
\nu(\delta) := \left( \lfloor \delta^{-1} \rfloor + 1  \right)!,  
$$ 
which is a super-decreasing function (here $\lfloor \cdot \rfloor$ is the floor function). 

Now, let $T$ be an arbitrary $K$-torus of dimension $d$ and having the minimal splitting field $P$. Then one can construct an exact sequence 
$$
1 \to T_1 \longrightarrow T_0 \longrightarrow T \to 1
$$
of $K$-tori and $K$-morphisms where $T_0$ is a product of $d$ copies of $\mathrm{R}_{P/K}(\mathbb{G}_m)$ (see \cite[Proposition 5.1]{RT}). One then uses the Nakayama-Tate Theorem  
(cf. \cite[Theorem 6 in \S 11.3]{Voskr}) to show that 
$$
[T(\mathbb{A}_S) : \overline{T(K)}] \ \ \text{divides} \ \ [T_0(\mathbb{A}_S) : \overline{T_0(K)}] \cdot \vert H^1(G(P/K) , X(T_1)) \vert, 
$$
where $X(T_1)$ is the group of characters of $T_1$ considered as a module over $G(P/K)$ (cf. \cite[\S 5]{RT}). It is well-known that the cohomology group $H^1(G(P/K) , X(T_1))$ is finite, but in fact one shows 
that its order always divides some number $\rho(d)$ that depends only on $d$. The first part of the proof implies that since $\mathfrak{d}_K(S \cap \mathrm{Spl}(P/K)) = \delta > 0$ we have that 
$$
[T_0(\mathbb{A}_S) : \overline{T_0(K)}]  \ \ \text{divides} \ \ \nu(\delta)^d. 
$$
So, the function $\phi(\delta , d) = \nu(\delta)^d \cdot \rho(d)$ is as required.  
\end{proof}

\begin{remark}
For our proof of Theorem \ref{T:Main} it would be sufficient to know that in the situation described in Theorem \ref{T:ASA-tori} the quotient $T(\mathbb{A}_S)/ \overline{T(K)}$ has finite exponent 
that divides some number that depends only on $\delta$ and $d$. This fact is somewhat easier to prove but the difference in the argument is not significant.  
\end{remark}

%
%
%
%

\section{Construction of generic independent tori} \label{S:GenT}

Let $G$ be an absolutely almost simple algebraic group defined over a field $K$, and let $M$ be the minimal Galois extension of $K$ over which $G$ becomes an inner form of the split group. Given a maximal $K$-torus $T$ of $G$, we let $\Phi(G , T)$ (resp., $W(G , T)$) denote the corresponding root system (resp., the Weyl group). Let $K_T$ denote the minimal splitting field of $T$ over $K$, and for an extension $L$ of $K$ we let $L_T$ denote the compositum $L\cdot K_T$. The action of the Galois group $G(L_T/L)$ on the characters of $T$ defines an injective homomorphism $\theta_{T,L} \colon G(L_T/L) \to \mathrm{Aut}(\Phi(G , T))$, and we will routinely identify $G(L_T/L)$ with its image. If $\im \theta_{T,L} \supset W(G , T)$ then $T$ is said to be {\it generic} over $L$.  Furthermore, we let $w$ denote the order of $W(G, T)$, and $c$ the number of conjugacy classes in it.

\begin{thm}\label{T:Gen} 
{\rm (\cite[Theorem 3.1]{PR-WC})} Assume that $K$ is a number field. Let $v_1, \ldots , v_c$ be $c$ nonarchimedean valuations of $K$ such that $G$ is $K_{v_i}$-split for all $i = 1, \ldots , c$. Furthermore, let $L$ be a finite Galois extension of $K$ such that $L \subset K_{v_i}$ for all $i = 1, \ldots , c$. Then for each $i \in \{1, \ldots , c\}$ one can choose a maximal $K_{v_i}$-torus $T_i^0$ of $G$ such that if a maximal $K$-torus $T$ of $G$ is $G(K_{v_i})$-conjugate to $T_i^0$ for all $i = 1, \ldots , c$ then $G(L_T/L) \supset W(G , T)$.   
\end{thm}

\vskip2mm 

We will need the following consequence of the theorem.

\begin{prop}\label{P:Gen} 
Fix $r \geq 1$ and suppose we are given $cr$ nonarchimedean valuations 
$$
v_1^{(1)}, \ldots , v_c^{(1)}, \ldots , v_1^{(r)}, \ldots , v_c^{(r)}
$$
of $K$ such that $G$ is $K_{v_i^{(j)}}$-split  for all $i = 1, \ldots , c$ and $j = 1, \ldots , r$. Then for each 
$i \in \{ 1, \ldots , c \}$ and $j, k \in \{ 1, \ldots , r\}$ one can choose a maximal $K_{v_i^{(j)}}$-torus $T_{ijk}^0$ of $G$ so that 
if $T_1, \ldots , T_r$ are maximal $K$-tori of $G$ such that $T_k$ is $G( K_{v_i^{(j)}})$-conjugate to $T_{ijk}^0$ for all $i = 1, \ldots , c$ 
and $j = 1, \ldots , r$ then the splitting fields $K_{T_1}, \ldots , K_{T_r}$ (which all contain $M$) satisfy the following properties:

\vskip2mm 

\begin{enumerate}[{\rm (1)}]
  
  \item for each $k = 1, \ldots , r$ we have $G(K_{T_k}/K) \supset W(G , T_k)$, and consequently \\  $G(K_{T_k}/M) = W(G , T_k)$; 
  
  \vskip1mm 
  
  \item $K_{T_1}, \ldots , K_{T_r}$ are linearly disjoint over $M$ (in which case $T_1, \ldots , T_r$ are said to be \newline  independent).

\end{enumerate}

\end{prop}
\begin{proof}
For $k = 1$, we choose the tori $T_{111}^0, \ldots , T_{c11}^0$ to be the tori $T_1^0, \ldots , T_c^0$ from Theorem \ref{T:Gen} for $v_i = v_i^{(1)}$ $(i = 1, \ldots , c)$, while the tori $T_{ij1}^0$ for $i = 1, \ldots , c$ and $j = 2, \ldots , r$ can be chosen arbitrarily. Suppose $1 \leq \ell < r$ and the tori
$T_{ijk}^0$ for $i = 1, \ldots , c$, $j = 1, \ldots , r$ and $k = 1, \ldots , \ell$ have already been selected, and describe how to choose these tori for $k = \ell + 1$. Let $\mathscr{T}_{ij}$ be 
a maximal $K_{v_i^{(j)}}$-split torus of $G$. 
We then define
$$
T_{ij(\ell+1)}^0 = \mathscr{T}_{ij} \ \ \text{for} \ \ i = 1, \cdots , c \ \ \text{and} \ \ j = 1, \ldots , \ell,
$$
with $T_{i(\ell+1)(\ell+1)}^0$ for $i = 1, \ldots , c$ being the tori $T_1^0, \ldots , T_c^0$ from Theorem \ref{T:Gen} constructed for $v_i = v_i^{(\ell+1)}$ for $i = 1, \ldots , c$.

Now, suppose we have maximal $K$-tori $T_1, \ldots , T_r$ of $G$ such that each $T_k$ is $G\left( K_{v_i^{(j)}} \right)$-conjugate to $T_{ijk}^0$ for $i = 1, \ldots , c$ and $j = 1, \ldots , r$. According to Theorem \ref{T:Gen}, for each $k = 1, \ldots , r$ and any finite Galois extension $L$ of $K$ such that $L \subset K_{v_i^{(k)}}$ for $i = 1, \ldots , c$ we have the inclusion
$$
G(L_{T_k}/L) \supset W(G , T_k). 
$$
In particular, $G(K_{T_k}/K) \supset W(G , T_k)$, hence $G(K_{T_k}/M) = W(G , T_k)$ (cf. \cite[Lemma 4.1(b)]{PR-WC}). So, we only need to prove that the splitting fields $K_{T_1}, \ldots , K_{T_r}$ are linearly disjoint over $M$. For this we will prove by induction on $d$, starting with $d = 0$, that
$$
[K_{T_{r-d}} \cdots K_{T_r} : M] = w^{d+1},
$$
where $w$ is the order of the Weyl group. There is nothing to prove if $d = 0$, so let us assume that $d > 0$ and for $L := K_{T_{r-d+1}} \cdots K_{T_r}$, the degree $[L : M] = w^d$. To complete the inductive argument, we will show that $[L_{T_{r-d}} : L] = w$. For this we note that by our construction each $T_k$ with $k \geq r - d + 1$ is $K_{v_i^{(r-d)}}$-isomorphic to $\mathscr{T}_{i(r-d)}$, yielding the inclusion
$$
K_{T_k} \subset K_{v_i^{(r-d)}} \ \ \text{for} \ \ i = 1, \ldots , c.
$$
Thus, $L \subset K_{v_i^{(r-d)}}$ for $i = 1, \ldots c$, implying by Theorem \ref{T:Gen} that $G(L_{T_k}/L) \supset W(G , T_k)$. Since $L \supset M$, we actually obtain that $[L_{T_k} : L] = w$, as required.

\end{proof}

\section{Proof of the Main Theorem}\label{S:Proof}

Set $\mathbb{S} := \mathrm{Spl}(M/K) \setminus (S \cap \mathrm{Spl}(M/K))$. We will consider two cases. 

\vskip2mm 

{\it Case 1.} $\mathbb{S}$ \underline{is finite}. In the terminology introduced in \cite{RT} this means that $S$ almost contains $\mathrm{Spl}(M/K)$, which is a {\it generalized arithmetic progression}. In this situation the triviality of $C^S(G)$ was established in \cite[Theorem B]{PR-CSP} and \cite[Theorem 1.5]{RT}.

\vskip2mm 

{\it Case 2.} $\mathbb{S}$ \underline{is infinite}. In view of Corollary \ref{C:0001}, to complete the proof of Theorem \ref{T:Main} in this case, it is enough to prove the following. 
\begin{prop}\label{P:Main} 
Assume that $\mathfrak{d}_K(S \cap \mathrm{Spl}(M/K)) > 0$  and $\mathbb{S}$ is infinite. Then there exists an integer $n \geq 1$ such that for every finite index normal subgroup $N$ of $\Gamma = G(\mathcal{O}(S))$ and any $y \in \overline{N} \cap \Gamma$ we have the inclusion $Z(N , y)(\overline{N} \cap \Gamma) \supset \Gamma^n$. 
\end{prop} 
Set $\omega = \mathfrak{d}_K(S \cap \mathrm{Spl}(M/K)) > 0$, and let $t = w$ be the order of the Weyl group of $G$. Choose $r$ so that 
$$
m^{-1} (1 - t^{-1})^r < \omega/2. 
$$
Then in the notations of Proposition \ref{P:Dens1}, we have 
$$
\theta := \mathfrak{d}_K(S \cap \mathrm{Spl}(M/K)) - m^{-1} (1 - t^{-1})^r > \omega / 2. 
$$
Let $n = \phi(\delta , d)$ in the notations of Theorem \ref{T:ASA-tori} where $\delta = \omega / (2r)$ and $d$ is the rank of $G$. We will show that this $n$ is as required.

\vskip1mm 

For $v \in V^K$ and a maximal $K_v$-torus $T$ of $G$, and let $T^{\small \mathrm{reg}}$ denote the Zariski-open subset of regular elements. It follows from the Implicit Function Theorem (cf. \cite[Proposition 3.4]{PRR}) that the map 
$$
\varphi_{v , T} \colon G(K_v) \times T^{\small \mathrm{reg}}(K_v) \longrightarrow G(K_v), \ \ \ (g , t) \mapsto gtg^{-1}, 
$$
is open; in particular, 
$$
\mathcal{U}(v , T) := \varphi_{v , T}(G(K_v) \times T^{\small \mathrm{reg}}(K_v))
$$
is an open subset of $G(K_v)$. It follows from the definition that $\mathcal{U}(v , T)$ is {\it conjugation-invariant}, and when $v$ is non-archimedean, it is also {\it solid}, i.e. it intersects every open subgroups of $G(K_v)$. 

Replace $N$ with a smaller finite index normal subgroup of $\Gamma$, we may assume that 
$$
\overline{N} = \prod_{v \notin S} N_v,$$ where $N_v \subset G(\mathcal{O}_v)$ is an open subgroup for  $v \notin S$. 
Furthermore, in view of \cite[Theorem 6.7]{PlRa}, we can find a finite subset $V \subset V^K \setminus S$ so that for $v \notin 
V^K \setminus (S \cup V)$ we have $N_v = G(\mathcal{O}_v)$ and $G$ is $K_v$-quasi-split. 
Let $r$ be as above, and let $c$ be the number of conjugacy classes of the Weyl group of $G$ as in \S\ref{S:GenT}. We pick $cr$ valuations $v_1^{(1)}, \ldots , v_c^{(1)}, \ldots , v_1^{(r)}, \ldots , v_c^{(r)} \in \mathbb{S} \setminus  V$. Then for each $v_{i}^{(j)}$, the group $G$ over  $K_{v_i^{(j)}}$ is quasi-split and at the same time an inner form of a split form (since $MK_{v_i^{(j)}} = K_{v_i^{(j)}}$). Thus, $G$ is  $K_{v_i^{(j)}}$-split for all $i = 1, \ldots , c$ and $j = 1, \ldots , r$. So, we are in a position to use Proposition \ref{P:0003} and find the maximal tori $T^0_{ijk}$ of $G$ $(i = 1, \ldots , c; \ j,k = 1, \ldots , r)$ with the properties described in this proposition. Set 
$$
V' = \{ v_1^{(1)}, \ldots , v_c^{(1)}, \ldots , v_1^{(r)}, \ldots , v_c^{(r)} \} \ \ \text{and} \ \  U_{ijk} = \mathcal{U}(v_i^{(j)} , T_{ijk}^0). 
$$
For each $k = 1, \ldots , r$, we define 
$$
W_k = \prod_{i,j} U_{ijk}
$$
which is an open subset of $\displaystyle G_{V'} = \prod_{v \in V'} G(K_v)$. Let us now show that 
\begin{equation}\label{E:0201}
\prod_{v \in V'} N_v \subset N \cdot W_k. 
\end{equation}
Indeed, since each $U_{ijk}$ is solid, we can find $s_{ijk} \in N_{v_i^{(j)}} \cap U_{ijk}$ and then  
set 
$$
\underline{s}_k = (s_{ijk})_{i,j} \in \prod_{v \in V'} N_v \ \ \text{for} \ \ k = 1, \ldots , r. 
$$ 
Then $W_k \underline{s}_k^{-1}$ is a neighborhood of the identity in $G_{V'}$, so the density of (diagonally embedded) $N$ in $\displaystyle \prod_{v \in V'} N_v$ implies that  
$$
\prod_{v \in V'} N_v \subset N \cdot (W_k \underline{s}_k^{-1}), 
$$    
and (\ref{E:0201}) follows. 
Now, fix $z\in \Gamma$ and show that 
\begin{equation}\label{E:0202}
z^n \in Z(N, y)(\overline{N} \cap \Gamma).  
\end{equation}
Since $V' \subset V^K \setminus (S \cup V)$, we have 
$$
z \in \prod_{v \in V'} G(\mathcal{O}_v) = \prod_{v \in V'} N_v, 
$$
so it follows from (\ref{E:0201}) that for each $k = 1, \ldots , r$ one can find $n_k \in N$ so that 
\begin{equation}\label{E:0205}
z_k := zn_k \in W_k. 
\end{equation}
Clearly, to prove (\ref{E:0202}), it is enough to show that 
\begin{equation}\label{E:0203}
z_k^n \in Z(N, y)(\overline{N} \cap \Gamma) \ \ \text{for at least one} \ \ k = 1, \ldots , r. 
\end{equation}
The argument used to prove (\ref{E:0201}) also proves the following.  
\begin{lemma}\label{L:0101}
Let $\mathscr{T}$ be a maximal $K$-torus of $G$, and let $\mathscr{V} \subset V^K \setminus S$ be a finite subset. For $v \in \mathscr{V}$, set $\Omega_v = 
\varphi_{v , \mathscr{T}}(N_v \times \mathscr{T}^{\small \mathrm{reg}}(K_v))$. Then for any $y \in \overline{N} \cap \Gamma$ there exists 
$a \in N$ such that 
$$
xa \in \prod_{v \in \mathscr{V}} \Omega_v,  
$$
so that we have  
$$
ya = g b g^{-1} 
$$
for some $g = (g_v)$ and $b = (b_v)$ with $g_v \in N_v$ and $b_v \in \mathscr{T}^{\small \mathrm{reg}}(K_v)$ for all $v \in \mathscr{V}$.  
\end{lemma}

\vskip2mm 

For each $k = 1, \ldots , r$, we let $\mathscr{T}_k = Z_G(z_k)^{\circ}$ denote the $K$-torus of $G$ containing $z_k$. Applying Lemma 
\ref{L:0101} to $\mathscr{T}_k$ and $\mathscr{V} = V \cup V'$, we find elements $a_k \in N$, $b_k \in \prod_{v \in \mathscr{V}} \mathscr{T}_k^{\small \mathrm{reg}}(K_v)$ and $g_k 
\in \prod_{v \in \mathscr{V}} N_v$ such that  
\begin{equation}\label{E:0103}
t_k := ya_k = g_k b_k g_k^{-1}. 
\end{equation}
It follows from (\ref{E:0205}) and (\ref{E:0103}) that the $K$-torus $T_k = Z_G(t_k)^{\circ}$ is $G(K_{v_i^{(j)}})$-conjugate to the torus $T_{ijk}^0$ for all $i = 1, \ldots , c$ and $j = 1, \ldots , r$. 
Let $P_k = K_{T_k}$ be the splitting field of $T_k$. It follows from Proposition \ref{P:Gen} that each $P_k$ is an extension of $M$ of degree $t = \vert W(G , T) \vert$ and together $P_1, \ldots , P_r$ are linearly disjoint over $M$. Taking into account our choice of $r$ and applying Proposition \ref{P:Dens1}, we obtain that there exists a $k \in \{1, \ldots , r\}$ such that 
$$
\bar{\frak{d}}_K(S \cap \mathrm{Spl}(P_k/K)) \geq \omega/(2r) = \delta.  
$$
Since the function $\phi$ in Theorem \ref{T:ASA-tori} is super-decreasing, we have that the index  
\begin{equation}\label{E:0206}
[T_k(\mathbb{A}_S) : \overline{T_k(K)}] \ \ \text{divides} \ \ n = \phi(\delta , d).   
\end{equation}
Obviously, $T_k(K)$ centralizes $t_k$ in $\widehat{G}$.  Since the centralizer $Z_{\widehat{G}}(t_k)$  is a closed subgroup of $\widehat{G}$ and the map $\pi$ is proper, hence closed, we conclude from (\ref{E:0206}) that 
\begin{equation}\label{E:0104}
\pi(Z_{\widehat{G}}(t_k)) \supset \overline{T_k(K)} \supset T_k(\mathbb{A}_S)^n. 
\end{equation} 
For $v \notin S$, we define 
$$
f_v = \left\{ \begin{array}{lll} (g_k)_v z_k^n (g_k)_v^{-1} & \text{if} & v \in \mathscr{V}, \\ 1 & \text{if} & v \notin \mathscr{V}. \end{array}   \right.
$$
It follows from (\ref{E:0103}) that the $S$-adele $f = (f_v)$ belongs to $T_k(\mathbb{A}_S)^n$. So, according to (\ref{E:0104}) that there exists $h \in Z_{\widehat{G}}(t)$ such that $\pi(h) = f$. In fact, $f \in \overline{\Gamma}$ (because $z \in \Gamma$ and $(g_k)_v \in N_v$ for all $v \in \mathscr{V}$), so $h \in \widehat{\Gamma}$, and therefore one can find $s \in \Gamma \cap h \widehat{N}$ is nonempty (as $\Gamma$ is dense in $\widehat{\Gamma}$ and $h\widehat{N}$ is open). Pick $b \in \Gamma \cap s\widehat{N}$. Then 
$$
[y , s] \in [t_k , s] \widehat{N} = [t_k , h]\widehat{N} = \widehat{N} 
$$ 
since $h \in Z_{\widehat{G}}(t_k)$. Thus, $[y , s] \in \Gamma \cap \widehat{N} = N$, and therefore $s \in Z(N , y)$. On the other hand, 
$$
s \in \pi(h) \overline{N} = f \overline{N} = z_k^n \overline{N}. 
$$
Indeed, for $v \in \mathscr{V}$ we have $f_vz_k^{-n} = [g_v , z_k^n] \in N_v$ since $g_v \in N_v$ for $v \in \mathscr{V}$, and for $v \in V^K \setminus (S \cup \mathscr{V})$ we have 
$f_v z_k^{-n} = z_k^{-n} \in G(\mathcal{O}_v) = N_v$, hence $fz_k^{-n} \in \overline{N}$, as required. Thus, $z_k^n \in Z(N, y)(\overline{N} \cap \Gamma)$, proving (\ref{E:0203}) and completing the proof of Proposition~\ref{P:Main}.  \hfill $\Box$  

\begin{remark}
While the overall structure of the proof of centrality in Cases 1 and 2 is similar, the argument in Case 1 relies on a different statement along the lines of Proposition~\ref{P:Main} (see \cite[\S3]{PR-CSP} and \cite[\S8]{RT}). For these reason, we decided to keep these cases separate. 
\end{remark}

\bibliographystyle{amsplain}

\end{document}